\documentclass[11pt]{amsart}
\usepackage{color}
\newtheorem{theorem}{Theorem}[section]
\newtheorem{lemma}[theorem]{Lemma}

\newtheorem{problem}[theorem]{Problem}

\theoremstyle{definition}
\newtheorem{definition}[theorem]{Definition}

\theoremstyle{remark}

\numberwithin{equation}{section}

\begin{document}

\title{ The $L_p$ Minkowski problem for $q$-torsional rigidity}


\author{Bin Chen}
\address{}
\curraddr{}
\email{chenb121223@163.com}
\thanks{}


\author{Xia Zhao}
\address{}
\curraddr{}
\email{zhaoxia20161227@163.com}
\thanks{}

\author{Weidong Wang}
\address{} 
\curraddr{}
\email{wangwd722@163.com}
\thanks{}

\author{Peibiao Zhao}
\address{} 
\curraddr{}
\email{pbzhao@njust.edu.cn}
\thanks{}
\thanks{}



\date{}

\dedicatory{}

\begin{abstract}
  In this paper, we introduce the so-called $L_p$ $q$-torsional measure for $p\in\mathbb{R}$ and $q>1$ by establishing the $L_p$ variational formula for the $q$-torsional
  rigidity of convex bodies without smoothness conditions.
  Moreover, we achieve the existence of solutions to the $L_p$ Minkowski problem $w.r.t.$ the $q$-torsional rigidity for discrete measure and general measure when $0<p<1$ and $q>1$.
\end{abstract}

\maketitle

\section{ Introduction}
The study of Brunn-Minkowski theory of convex bodies (i.e., a
compact, convex set) in Euclidean spaces $\mathbb{R}^n$ has been an
active field in the past century, which is developed from a few
basic concepts: support functions, Minkowski combinations, and mixed
volumes. As we all know, the Minkowski problem is one of the main
cornerstones in Brunn-Minkowski theory of convex bodies.

The classical Minkowski problem  is to find a convex body $K$ with
the prescribed surface area measure $S(K,\cdot)$, which is induced
 by the volume
variational, that is, for each convex body $L$, there holds
\begin{align}\label{1.0}
\frac{d}{dt}V(K+tL)\bigg|_{t=0}=\int_{\mathbb{S}^{n-1}}
h(L,\cdot)dS(K,\cdot),
\end{align}
where $K+tL$ is the Minkowski sum, and $h(L,\cdot)$ is the support function of $L$.

The solution of the classical Minkowski problem has been solved by
famous mathematicians such as Minkowski \cite{M,M1}, Alexandrov
\cite{A,A0}, Fenchel and Jessen \cite{FJ}, Lewy \cite{L0} and
Nirenberg \cite{N}.

As an extension of the classical Minkowski problem, the $L_p$
Minkowski problem w.r.t. the $L_p$ surface area measure
$S_p(K,\cdot)$ of a convex body $K$ containing the origin in its
interior was proposed and  studied in \cite{L}. Here $S_p(K,\cdot)$
is defined  by the volume variational under the Firey's $p$-sum
(\cite{F})
\begin{align}\label{1.01}
\frac{d}{dt}V(K+_pt\cdot L)\bigg|_{t=0}&=\frac{1}{p}\int_{\mathbb{S}^{n-1}}
h(L,\cdot)^ph(K,\cdot)^{1-p}dS(K,\cdot)\\
\nonumber&\hat{=}\frac{1}{p}\int_{\mathbb{S}^{n-1}}
h(L,\cdot)^pdS_p(K,\cdot),
\end{align}
for a compact convex set $L$ containing the origin and $p\geq1$.
Obviously, the case of $p=1$ is the formula (\ref{1.0}). Since then,
the $L_p$ Minkowski problem has become an interest central object in
convex geometric analysis and has been widely considered, see e.g.,
\cite{C,CH,C1,HLYZ,L1,L2,Lu,S0,Z}. Moreover, there are various
versions of Minkowski problems related to other functionals in
Brunn-Minkowski theory, for instance, the dual Minkowski problem
\cite{Hu}, the logarithmic Minkowski problem \cite{Bo} and the
Orlicz-Minkowski problem \cite{GH,H,LL}.

From the  statements above, one knows that the difference in
different geometric functionals usually derives some  new and
different geometric measures. In recent years, some geometric
functionals with physical backgrounds have been introduced into the
Brunn-Minkowski theory, and related Minkowski-type problems have
also been gradually studied, see e.g., \cite{C2,C3,J,ZX}. One of
them is the $q$-torsional rigidity, which is exactly the geometrical
functional concerned in the present paper.

For convenience, let $\mathcal{K}_o^n$ be the set of convex bodies
containing the origin $o$ in their interiors, and $C_+^2$ be the
class of the convex body of $C^2$ if its boundary has the positive
Gauss curvature.

Now, we recall the concept of the $q$-torsional rigidity. Let
$\Omega$ be the interior of convex body $K$ in $\mathbb{R}^n$ and
$\overline{\Omega}=K$. For $q>1$, the $q$-torsional rigidity
$T_{q}(K)$ is defined by (see \cite{C4})
\begin{align}\label{1.1}
\frac{1}{T_q(K)}=\inf\bigg\{\frac{\int_\Omega|
\nabla\phi|^qdx}{[\int_\Omega|\phi|dx]^q}:
\phi\in W^{1,q}_0(\Omega), \int_\Omega\phi dx>0\bigg\}.
\end{align}
The functional defined in (\ref{1.1}) admits a minimizer $\varphi\in W^{1,q}_0(\Omega)$,
and $c\varphi$ (for some constant $c$) is the unique positive solution of the following boundary value problem (see \cite{BK} or \cite{HZ})
\begin{align}\label{1.2}
\begin{cases}
\triangle_{q}\varphi=-1\ \ \ \ in\ \Omega,\\
\varphi=0\ \ \ \ \ on\ \partial\Omega,
\end{cases}
\end{align}
where $\triangle_q$ is the $q$-Laplace operator.

Following this, Huang et al. (\cite{Hu1}) defined the $q$-torsional measure as
$$\mu_q^{tor}(\Omega,\eta)=\int_{g^{-1}(\eta)}|\nabla\varphi|^qd\mathcal{H}^{n-1},$$
for Borel set $\eta\subseteq\mathbb{S}^{n-1}$. Here $g: \partial\Omega\rightarrow\mathbb{S}^{n-1}$ is the Gauss map, and $\mathcal{H}^{n-1}$ is the $(n-1)$-dimensional Hausdorff measure.

Meanwhile, they replaced the volume functional with the
$q$-torsional rigidity in (\ref{1.0}), and established the following
variational formula in smooth case: Let $K,L$ be two convex bodies
of class $C_+^2$, and $q>1$, then
\begin{align}\label{1.3}
\frac{d}{dt}T_q(K+tL)\bigg|_{t=0}=(q-1)
T_q(K)^{\frac{q-2}{q-1}}
\int_{\mathbb{S}^{n-1}}h(L,\xi)d\mu_{q}^{tor}(K,\xi),
\end{align}
and the $q$-torsional rigidity formula
\begin{align}\label{1.4}
T_q(K)^{\frac{1}{q-1}}=\frac{q-1}{q+n(q-1)}
\int_{\mathbb{S}^{n-1}}h(K,\xi)d\mu_q^{tor}(K,\xi).
\end{align}
Thus $\mu_q^{tor}(K,\cdot)$ can be regarded as induced by the formula (\ref{1.3}).

Inspired by the important role of the $L_p$ volume variational
formula  in the Brunn-Minkowski theory, we now establish the $L_p$
variational formula $w.r.t.$ the $q$-torsional rigidity (For a
detailed proof, see Theorem \ref{the4.1}) below.

\begin{theorem}\label{the1.0.1}
Let $K\in\mathcal{K}_o^n$, $1\leq p<\infty$ and $q>1$. If $L$ is a compact convex set containing the origin, then
$$\frac{d}{dt}T_q(K+_pt\cdot L)\bigg|_{t=0}=\frac{q-1}{p}
T_q(K)^{\frac{q-2}{q-1}}
\int_{\mathbb{S}^{n-1}}h(L,\xi)^ph(K,\xi)^{1-p}d\mu_q^{tor}(K,\xi).$$
\end{theorem}

Similar to the definition of $L_p$ surface area measure, we can
define the so-called $L_p$ $q$-torsional measure as follows.

\begin{definition}\label{def1.1}
{ \it For $K\in\mathcal{K}_o^n$, $p\in\mathbb{R}$ and $q>1$, the
$L_p$ $q$-torsional measure $\mu^{tor}_{p,q}(K,\cdot)$ of $K$ is
defined by
$$\mu^{tor}_{p,q}(K,\eta)=\int_\eta h(K,\xi)^{1-p}d\mu^{tor}_q(K,\xi)$$
for each Borel set $\eta\subseteq\mathbb{S}^{n-1}$.
Obviously, the case $p=1$ is just the $q$-torsional measure.}
\end{definition}

Notice that the $q$-torsional measure induced in (\ref{1.3}) has so
far not been extended to the case of any convex body without
smoothness conditions. To implement the variational formula in
Theorem \ref{the1.0.1}, we first need to  deal with the weak
convergence of the $q$-torsional measure, and generalize the formula
(\ref{1.3}) to any convex body without smoothness (see Section 3 and
Section 4 for details).

Naturally, the $L_p$ Minkowski problem of the $q$-torsional rigidity
can be proposed  as below

\begin{problem}\label{pro1.1}
(The $L_p$ Minkowski problem of $q$-torsional rigidity)\ \  Suppose $\mu$ is a finite Borel measure on $\mathbb{S}^{n-1}$, $p\in\mathbb{R}$ and $q>1$. What are the necessary and sufficient conditions on $\mu$ such that $\mu$ is the $L_p$ $q$-torsional measure $\mu_{p,q}^{tor}(K,\cdot)$ of a convex body $K\in\mathcal{K}_o^n$?
\end{problem}

For Problem \ref{pro1.1}, Sun-Xu-Zhang \cite{SX} only proved the
uniqueness of solutions with $p\geq1$ and $q>1$ in smooth case, but
the existence of solutions was not solved. When $p>1$ and $q=2$,
Chen-Dai \cite{CD} obtained the existence and uniqueness of
solutions. For the classical case $p=1$, Colesanti-Fimiani \cite{C2}
proved the existence and uniqueness of solutions when $q=2$.

The main aim of the present paper is to consider the existence of
solutions to Problem \ref{pro1.1}. We first present a solution to
Problem \ref{pro1.1} for discrete measures when $0<p<1$ and $q>1$
(see Theorem \ref{the5.3.1}) as follows

\begin{theorem}\label{the1.0}
Let $\mu$ be a finite positive Borel measure on $\mathbb{S}^{n-1}$
which is not concentrated on any closed hemisphere. If $\mu$ is discrete,
$0<p<1$ and $q>1$. Then there exists a convex polytope $P$ containing the origin in its
interior such that $\mu_{p,q}^{tor}(P,\cdot)=\mu$.
\end{theorem}

Theorem \ref{the1.0} will yield a solution to Problem \ref{pro1.1}
for general measures when $0<p<1$ and $q>1$ (see Theorem
\ref{the6.1}) as below.

\begin{theorem}\label{the1.1}
Let $\mu$ be a finite positive Borel measure on $\mathbb{S}^{n-1}$ which is not concentrated on any closed hemisphere. Suppose $0<p<1$ and $q>1$. Then there exists a convex body $K\in\mathcal{K}_o^n$ such that $\mu_{p,q}^{tor}(K,\cdot)=\mu$.
\end{theorem}

The organization of this paper is as follows.
The background materials and some results are introduced in Section 2.
In Section 3, we show that the formula (\ref{1.3}) is valid for general convex bodies without smoothness condition.
In Section 4, based on the conclusions in Sect. 3, the $L_p$ Hadamard variational formula for $q$-torsional rigidity is established.
In Sections 5 and 6, we prove the existence of solutions to the $L_p$ Minkowski problem of $q$-torsional rigidity for discrete measure and general measure when $0<p<1$ and $q>1$.\\

\section{ Preliminaries}
In this section, we introduce some necessary facts about convex bodies that readers can refer to  good books of Gardner \cite{G} and  Schneider \cite{S}.\\

\subsection{Basic facts on convex bodies}
\vglue-10pt
 \indent

For $K\in\mathcal{K}_o^n$ in $\mathbb{R}^n$.
The support function $h(K,\cdot): \mathbb{S}^{n-1}\rightarrow\mathbb{R}$ of $K$ is defined by
$$h(K,\zeta)=\max\{\zeta\cdot Y: Y\in K\},\ \ \zeta\in\mathbb{S}^{n-1},$$
where $``\cdot"$ for the standard inner product in $\mathbb{R}^n$.
We clearly know that the support function is a homogeneous convex function with degree $1$.
Let $K$ and $L$ be convex bodies, the Minkowski combination of $K$ and $L$ is defined by
$$K+tL=\{x+ty: x\in K, \ y\in L\}$$
for $t>0$. The $K+tL$ is a convex body whose support function is given by
$$h(K+tL,\cdot)=h(K,\cdot)+th(L,\cdot).$$

For $K,L\in\mathcal{K}_o^n$ and $p\geq1$, the Firey's $p$-sum, $K+_pt\cdot L$, of $K$ and $L$ is defined through its support function
$$h(K+_pt\cdot L,\cdot)^p=h(K,\cdot)^p+th(L,\cdot)^p.$$
In particular, when $p=1$, the Firey's $p$-sum is just the classical Minkowski sum.

The radial function $\rho_K: \mathbb{S}^{n-1}\rightarrow(0,\infty)$ of $K$ is defined by
$$\rho(K,\upsilon)=\max\{c>0: c\upsilon\in K\}, \ \ \upsilon\in\mathbb{S}^{n-1}.$$
The radial map $r_K: \mathbb{S}^{n-1}\rightarrow\partial K$ is
$$r_K(\upsilon)=\rho_K(\upsilon)\upsilon,$$
for $\upsilon\in\mathbb{S}^{n-1}$, i.e. $r_K(\upsilon)$ is the unique on $\partial K$ located on the ray parallel to $\upsilon$ and emanating from the origin.

Two convex bodies $K,L\in\mathcal{K}_o^n$ are said to be homothetic if
$K=\lambda L+x$ for some constant $\lambda>0$ and $x\in\mathbb{R}^n$,
particularly, $K$ and $L$ are said to be dilates of each other if $K=\lambda L$.
The Hausdorff metric on $\mathcal{K}_o^n$, $d_{\mathcal{H}}(\cdot,\cdot)$, is used to measure the distance
between two convex bodies $K,L\in\mathcal{K}_o^n$, and it is defined by
$$d_{\mathcal{H}}(K,L)=\max_{\zeta\in\mathbb{S}^{n-1}}|h(K,\zeta)-h(L,\zeta)|
=\|h(K,\cdot)-h(L,\cdot)\|_\infty.$$

Denote by $C(\mathbb{S}^{n-1})$ the set of continuous functions defined on $\mathbb{S}^{n-1}$, which is equipped
with the metric induced by the maximal norm. Write $C_+(\mathbb{S}^{n-1})$ for the set of strictly positive functions in $C(\mathbb{S}^{n-1})$.

For a convex body $K$ and $\zeta\in\mathbb{S}^{n-1}$, the support hyperplane $H(K,\zeta)$ is defined by
$$H(K,\zeta)=\{Y\in\mathbb{R}^n: \zeta\cdot Y=h(K,\zeta)\}.$$
The half-space $H^-(K,\zeta)$ in the direction $\zeta$ is defined by
$$H^-(K,\zeta)=\{Y\in\mathbb{R}^n: \zeta\cdot Y\leq h(K,\zeta)\}.$$
The support set $\mathcal{F}(K,\zeta)$ in the direction $\zeta$ is defined by
$$\mathcal{F}(K,\zeta)=K\cap H(K,\zeta).$$

For a compact set $K\in\mathbb{R}^n$, the diameter of $K$ is defined by
$$d(K)=\max\{|X-Y|: X,Y\in K\}.$$
\

\subsection{Aleksandrov bodies}
\vglue-10pt
 \indent

Given a function $f\in C_+(\mathbb{S}^{n-1})$, let $K\subset\mathbb{R}^n$ be such
\begin{align}\label{2.1}
\overline{K}:=\bigcap_{\xi\in\mathbb{S}^{n-1}}\{X\in\mathbb{R}^n: X\cdot\xi\leq f(\xi)\}.
\end{align}
Since $f$ is both positive and continuous, $K$ must be convex body in $\mathbb{R}^n$ that contain the origin. The $K$ is often called the Aleksandrov body associated with $f$. For Aleksandrov body $K$ associated with $h(K,\cdot)$, we see that
$$h(K,\cdot)\leq f.$$
Let
$$\omega_h=\{\xi\in\mathbb{S}^{n-1}: h(K,\xi)<f(\xi)\}.$$
A basic fact established by Aleksandrov is that
$$S(K,\omega_h)=0,$$
where $S(K,\cdot)$ is the surface area measure on $\partial K$ defined by the $(n-1)$-Hausdorff measure. Consequently,
\begin{align}\label{2.2}
h(K,\cdot)=f,\ \ a.e.\ \ with\ \ respect \ \ to \ \ S(K,\cdot).
\end{align}
Obviously, if $f$ is the support function of $K\in\mathcal{K}_o^n$, then $K$ itself is the Aleksandrov body associated with $f$.

{\it Aleksandrov's Convergence Lemma:} if the functions $h_i\in C_+(\mathbb{S}^{n-1})$
have associated Aleksandrov bodies $K_i\in\mathcal{K}_o^n$, then
$$h_i\rightarrow h\in C_+(\mathbb{S}^{n-1})\ uniformly\ \ \Rightarrow\ \ K_i\rightarrow K\ in\  the\  Hausdorff\  metric,$$
where $K$ is the Aleksandrov body associated with $h$.

\section{ Variational formula for $q$-torsional rigidity of general convex bodies}
\
In this part, we prove that the formula (\ref{1.3}) holds for general convex bodies
without smoothness condition. Before this, we first need to obtain the weak convergence of $q$-torsional measure.
\\

\subsection{Basic facts on $q$-torsional rigidity}
\vglue-10pt
 \indent

The following basic facts and geometric inequalities about $q$-torsional rigidity can be referred to \cite{Hu1} or \cite{SX}.

\begin{lemma}\label{lem3.1}
{\it Let $K,L$ be two convex bodies of class $C_+^2$ and $q>1$. The $q$-torsional rigidity has the following properties:

(1) It is positively homogeneous of degree $(q+n(q-1))$, i.e., $T_q(sK)=s^{q+n(q-1)}T_q(K)$ for $s>0$.

(2) It is monotonically increasing, that is, $T_q(K)\leq T_q(L)$ with $K\subseteq L$.

(3) It is translation invariant, that is, $T_q(K+x)=T_q(K)$, $x\in\mathbb{R}^n$.}
\end{lemma}

If we multiply $\frac{1}{q+n(q-1)}$ on the integral in the right hand side of (\ref{1.3}), the mixed $q$-torsional rigidity defined as
\begin{align}\label{3.1}
T_q(K,L)=\frac{q-1}{q+n(q-1)}T_q(K)^{\frac{q-2}{q-1}}
\int_{\mathbb{S}^{n-1}}h(L,\xi)d\mu_q^{tor}(K,\xi).
\end{align}
When $L=K$, it reduces to the $q$-torsion rigidity formula (\ref{1.4}), i.e. $T_q(K,K)=T_q(K)$.

The Brunn-Minkowski inequality for $q$-torsional rigidity reads: If $K,L$ are convex bodies of class $C_+^2$ and $q>1$, then
\begin{align}\label{3.2}
T_q(K+L)^{\frac{1}{q+n(q-1)}}\geq T_q(K)^{\frac{1}{q+n(q-1)}}
+T_q(L)^{\frac{1}{q+n(q-1)}},
\end{align}
with equality if and only if $K$ and $L$ are homothetic. Moreover, the Brunn-Minkowski inequality yields the Minkowski inequality of $q$-torsional rigidity,
\begin{align}\label{3.3}
T_q(K,L)^{q+n(q-1)}\geq T_q(K)^{q+n(q-1)-1}T_{q}(L),
\end{align}
with equality if and only if $K$ and $L$ are homothetic.\\

\subsection{Weak convergence of $q$-torsional measure}
\vglue-10pt
 \indent

For $\Omega$ be an open convex bounded set, $\Omega_i (i\in\mathbb{N})$ be a sequence of open convex bounded sets in $\mathbb{R}^n$ and $\varphi_i, \varphi\in W_0^{1,q}(\Omega)$.
Let $K=\overline{\Omega}$ and $K_i=\overline{\Omega}_i$, and let, for $\mathcal{H}^{n-1}$-a.e $\xi\in\mathbb{S}^{n-1}$
$$\hbar(\xi)=|\nabla\varphi(r_{K}(\xi))|(\mathcal{J}(\xi))^{\frac{1}{q}},
\ \ \ \ \hbar_i(\xi)=|\nabla\varphi_i(r_{K_i}(\xi))|(\mathcal{J}_i(\xi))^{\frac{1}{q}},$$
where $\mathcal{J}, \mathcal{J}_i$ are the Jacobian functions introduced in following lemma.

\begin{lemma}\label{lem3.2}\cite{J}
Let $\Omega$ be an open convex bounded set that contains the origin $o$, $K=\overline{\Omega}$ and $f: \partial K\rightarrow\mathbb{R}$ be $\mathcal{H}^{n-1}$-integrable. Then
$$\int_{\partial K}f(x)d\mathcal{H}^{n-1}(x)
=\int_{\mathbb{S}^{n-1}}f(r_K(\xi))\mathcal{J}(\xi)d\xi,$$
where $\mathcal{J}$ is defined $\mathcal{H}^{n-1}$-a.e. on $\mathbb{S}^{n-1}$ by
$$\mathcal{J}(\xi)=\frac{(\rho_K(\xi))^n}{h_K(g_K(r_K(\xi)))}.$$
Moreover, there exist constants $c_1,c_2>0$ such that $c_1<\mathcal{J}(\xi)<c_2$ for $\mathcal{H}^{n-1}$-a.e. $\xi\in\mathbb{S}^{n-1}$. Furthermore, assume that $\{K_i\}_{i\in\mathbb{N}}$ is a sequence of bounded convex bodies converging to $K$ w.r.t. to the Hausdorff metric. Define $\mathcal{J}_i: \mathbb{S}^{n-1}\rightarrow(0,\infty)$
$$\mathcal{J}_i(\xi)=\frac{(\rho_{K_i}(\xi))^n}{h_{K_i}(g_{K_i}(r_{K_i}(\xi)))},\ \ i\in\mathbb{N}.$$
Then there exists $i_0\geq1$ such that if $i\geq i_0$, then $\mathcal{J}_i(\xi)$ is bounded from below and above, uniformly w.r.t. $\xi$ and $i$, and $\{\mathcal{J}_i\}$ converge to $\mathcal{J}$, $\mathcal{H}^{n-1}$-a.e. on $\mathbb{S}^{n-1}$.
\end{lemma}

For functions $\varphi_i, \varphi\in W_0^{1,q}(\Omega)$,
we consider replacing $q$-equilibrium potential of $\Omega$ in lemma 4.6 in \cite{C3} with $\varphi$, and give the following lemma.

\begin{lemma}\label{lem3.3}
Suppose $q>1$, then
$$\lim_{i\rightarrow\infty}\int_{\mathbb{S}^{n-1}}
|\hbar_i^q(\xi)-\hbar^q(\xi)|d\xi=0.$$
\end{lemma}
\noindent The proof of this lemma is very similar to Lemma 4.6 in \cite{C3} and thus will be omitted.

Using the above results,
we now establish the weak convergence of $q$-torsional measure.
\begin{theorem}\label{th3.6}
 Let $\Omega$ be an open convex bounded set, $\Omega_i (i\in\mathbb{N})$ be a sequence of open convex bounded sets in $\mathbb{R}^n$.
If $q>1$, and $\overline{\Omega}_i$ converges to $\overline{\Omega}$ in the Hausdorff metric as $i\rightarrow0$.
Then the sequence of measures $\mu_q^{tor}(\Omega_i,\cdot)$ weakly converges to $\mu_q^{tor}(\Omega,\cdot)$.
\end{theorem}

\begin{proof}
Let $\rho_i, \rho$ and $g_i, g$ be the radial functions and Gauss maps of $\overline{\Omega}_i, \overline{\Omega}$, respectively.
In order to prove that $\mu_q^{tor}(\Omega_i,\cdot)$ weakly converges to $\mu_q^{tor}(\Omega,\cdot)$, it is sufficient to show that for any continuous function $f$ on $\mathbb{S}^{n-1}$ satisfies
\begin{align}\label{3.2.1}
\lim_{i\rightarrow\infty}\int_{\mathbb{S}^{n-1}}f(\xi)
d\mu_q^{tor}(\Omega_i,\xi)=\int_{\mathbb{S}^{n-1}}f(\xi)
d\mu_q^{tor}(\Omega,\xi).
\end{align}
From the definitions of $\mu_q^{tor}(\Omega,\cdot)$ and $\hbar(\xi)$, and Lemma \ref{lem3.2}, then (\ref{3.2.1}) is equivalent to
\begin{align}\label{3.4}
\lim_{i\rightarrow\infty}\int_{\mathbb{S}^{n-1}}
f(g_i(\rho_i(\xi)\xi))\hbar_i^q(\xi)d\xi=
\int_{\mathbb{S}^{n-1}}f(g(\rho(\xi)\xi))\hbar^q(\xi)d\xi.
\end{align}
Note that
\begin{align*}
&\bigg|\int_{\mathbb{S}^{n-1}}f(g_i(\rho_i(\xi)\xi))\hbar_i^q(\xi)d\xi
-f(g(\rho(\xi)\xi))\hbar^q(\xi)d\xi\bigg|\\
&\leq\bigg|\int_{\mathbb{S}^{n-1}}f(g_i(\rho_i(\xi)\xi))
[\hbar_i^q(\xi)-\hbar^q(\xi)]d\xi\bigg|\\
&\ \ \ +\bigg|\int_{\mathbb{S}^{n-1}}[f(g_i(\rho_i(\xi)\xi))
-f(g(\rho(\xi)\xi))]\hbar^q(\xi))d\xi\bigg|.
\end{align*}
Since $f$ is continuous on $\mathbb{S}^{n-1}$, then, by Lemma \ref{lem3.3}
$$\lim_{i\rightarrow\infty}\bigg|\int_{\mathbb{S}^{n-1}}f(g_i(\rho_i(\xi)\xi))
[\hbar_i^q(\xi)-\hbar^q(\xi)]d\xi\bigg|=0.$$
Using that $g_i$ converges to $g$ almost everywhere on $\mathbb{S}^{n-1}$ (see Remark 3.5 in \cite{C2}), and $\rho_i$ converges to $\rho$ uniformly, one has $g_i(\rho_i(\xi)\xi)$ converges to $g(\rho(\xi)\xi)$ almost everywhere on $\mathbb{S}^{n-1}$ as $i\rightarrow\infty$. Thus we have
$$\lim_{i\rightarrow\infty}\bigg|\int_{\mathbb{S}^{n-1}}[f(g_i(\rho_i(\xi)\xi))
-f(g(\rho(\xi)\xi))]\hbar^q(\xi))d\xi\bigg|=0.$$
Hence, we obtain (\ref{3.4}). This completes the proof.
\end{proof}
\

\subsection{Variational formula for $q$-torsional rigidity of general convex bodies}
\vglue-10pt
 \indent

For $h\in C_+(\mathbb{S}^{n-1})$, denote by $T_q(h)$ the $q$-torsional rigidity of a Aleksandrov body associated with $h$. Since the Aleksandrov body associated with the support function $h_K$ of a  $K\in\mathcal{K}_o^n$ is the $K$ itself, we have
\begin{align}\label{3.8}
T_q(h_K)=T_q(K).
\end{align}
Let $I\subset \mathbb{R}$ be an interval containing $0$ and suppose that
$$h_t(\xi)=h(t,\xi): I\times \mathbb{S}^{n-1}\rightarrow(0,\infty)$$
is continuous.

For fixed $t\in I$, let $K_t\subset\mathbb{R}^n$ be such that
$$\overline{K}_t=\bigcap_{\xi\in\mathbb{S}^{n-1}}\{x\in\mathbb{R}^n: x\cdot\xi\leq h(t,\xi)\}.$$
This is the Aleksandrov body associated with $h_t$. The family of convex domains $\{K_t\}_{t\in I}$ will be called the family of Aleksandrov bodies with $h_t$. Obviously, from (\ref{2.2}) we have, for each $t\in I$,
\begin{align}\label{3.9}
h({K_t},\cdot)\leq h_t \ \ and \ \ h({K_t},\cdot)=h_t,\ \ a.e.\ \ with \ \ respect\ \ to \ \ S_{K_t}.
\end{align}

From (\ref{3.1}) and Theorem \ref{th3.6}, if $q>1$, we have for the $q$-torsional rigidity formula of $K\in\mathcal{K}^n_o$,
\begin{align}\label{3.10}
T_q(K)^{\frac{1}{q-1}}=\frac{q-1}{q+n(q-1)}
\int_{\mathbb{S}^{n-1}}h(K,\xi)d\mu_q^{tor}(K,\xi).
\end{align}
Using (\ref{3.3}) and Theorem \ref{th3.6}, we obtain
\begin{align}\label{3.11}
T_q(K,L)^{q+n(q-1)}\geq T_q(K)^{q+n(q-1)-1}T_q(L),
\end{align}
for all convex bodies $K,L\in\mathcal{K}^n_o$.

The proof of following lemma regarding the variation of $q$-torsional rigidity is similar to that of its analogue for volume or $q$-capacity (see \cite{S}, Lemma 6.5.3 or \cite{C3}, Lemma 5.1). For reader's convenience, we list the proof process.

\begin{lemma}\label{lem3.7}
{\it Let $I\subset\mathbb{R}$ be an interval containing both $0$ and some positive number and let
$$h(t,\xi): I\times\mathbb{S}^{n-1}\rightarrow(0,\infty)$$
be continuous and such that the convergence in
\begin{align}\label{3.12}
h_+^\prime(0,\xi)=\lim_{t\rightarrow0^+}\frac{h(t,\xi)-h(0,\xi)}{t}
\end{align}
is uniform on $\mathbb{S}^{n-1}$. If $\{K_t\}_{t\in I}$ is the family of Aleksandrov bodies associated with $h_t$, and $q>1$, then
$$\lim_{t\rightarrow0^+}\frac{T_q(K_t)-T_q(K_0)}{t}
=(q-1)T_q(K_0)^{\frac{q-2}{q-1}}
\int_{\mathbb{S}^{n-1}}h_+^\prime(0,\xi)d\mu_q^{tor}(K_0,\xi).$$}
\end{lemma}

\begin{proof}
The uniform convergence of (\ref{3.12}) implies that $h_t\rightarrow h_0$, uniformly on $\mathbb{S}^{n-1}$. Thus, by the Aleksandrov convergence lemma,
\begin{align}\label{3.13}
K_t\rightarrow K_0
\end{align}
in the Hausdorff metric.
From Theorem \ref{th3.6}, the $q$-torsional measure $\mu_q^{tor}(K_t,\cdot)$ converges weakly to $\mu_q^{tor}(K_0,\cdot)$, we obtain
\begin{align}\label{3.14}
\lim_{t\rightarrow 0^+}\int_{\mathbb{S}^{n-1}}\frac{h(t,\xi)-h(0,\xi)}
{t}d\mu_q^{tor}(K_t,\xi)
=\int_{\mathbb{S}^{n-1}}h_+^\prime(0,\xi)d\mu_q^{tor}(K_0,\xi).
\end{align}
So, (\ref{3.9}), (\ref{3.10}) and the fact that $\mu_q^{tor}(K,\cdot)$ is absolutely continuous with respect to $S_K$, imply that
\begin{align}\label{3.15}
\nonumber T_q(K_t)&=\frac{q-1}{q+n(q-1)}
T_q(K_t)^{\frac{q-2}{q-1}}
\int_{\mathbb{S}^{n-1}}h({K_t},\xi)d\mu_q^{tor}(K_t,\xi)\\
&\geq\frac{q-1}{q+n(q-1)}
T_q(K_t)^{\frac{q-2}{q-1}}
\int_{\mathbb{S}^{n-1}}h_t(\xi)d\mu_q^{tor}(K_t,\xi)
\end{align}
From (\ref{3.15}), the definition of mixed $q$-torsional rigidity, (\ref{3.9}) at $t=0$, we have
\begin{align*}
&\lim_{t\rightarrow 0^+}\frac{T_q(K_t)-T_q(K_t,K_0)}{t}\\
&\ \ =\frac{q-1}{q+n(q-1)}\lim_{t\rightarrow 0^+}T_q(K_t)^{\frac{q-2}{q-1}}
\int_{\mathbb{S}^{n-1}}\frac{h_t(\xi)
-h({K_0},\xi)}{t}d\mu_q^{tor}(K_t,\xi)\\
&\ \ \geq\frac{q-1}{q+n(q-1)}\lim_{t\rightarrow 0^+}T_q(K_t)^{\frac{q-2}{q-1}}
\int_{\mathbb{S}^{n-1}}\frac{h_t(\xi)
-h_0(\xi)}{t}d\mu_q^{tor}(K_t,\xi).
\end{align*}
By (\ref{3.13}), we have $\lim_{t\rightarrow0^+}T_q(K_t)=T_q(K_0)$.
When conbined with (\ref{3.14}), gives
\begin{align}\label{3.16}
&\nonumber\lim_{t\rightarrow 0^+}\frac{T_q(K_t)-T_q(K_t,K_0)}{t}\\
&\ \ \geq\frac{q-1}{q+n(q-1)}
T_q(K_0)^{\frac{q-2}{q-1}}
\int_{\mathbb{S}^{n-1}}h_+^\prime(0,\xi)d\mu_q^{tor}(K_0,\xi).
\end{align}
We set
$$l=\frac{q-1}{q+n(q-1)}
T_q(K_0)^{\frac{q-2}{q-1}}
\int_{\mathbb{S}^{n-1}}h_+^\prime(0,\xi)d\mu_q^{tor}(K_0,\xi),$$
So (\ref{3.16}) and (\ref{3.11}) show that
\begin{align*}
l&\leq\lim_{t\rightarrow 0^+}\frac{T_q(K_t)-T_q(K_t,K_0)}{t}\\
&\leq\lim_{t\rightarrow 0^+}\frac{T_q(K_t)-T_q(K_t)^{1-\frac{1}
{q+n(q-1)}}T_q(K_0)^{\frac{1}{q+n(q-1)}}}{t}.
\end{align*}
However, (\ref{3.13}) gives $\lim_{t\rightarrow0^+}T_q(K_t)=T_q(K_0)$, and hence
\begin{align}\label{3.17}
l\leq T_q(K_0)^{1-\frac{1}{q+n(q-1)}}\lim_{t\rightarrow 0^+}
\frac{T_q(K_t)^{\frac{1}{q+n(q-1)}}
-T_q(K_0)^{\frac{1}{q+n(q-1)}}}{t}.
\end{align}

On the other hand, by the inequality in (\ref{3.9}) and the uniform convergence in (\ref{3.12}), we have
\begin{align*}
&\lim_{t\rightarrow0^+}\frac{T_q(K_0,K_t)-T_q(K_0)}{t}\\
&\ \ =\frac{q-1}{q+n(q-1)}\lim_{t\rightarrow0^+}
T_q(K_0)^{\frac{q-2}{q-1}}
\int_{\mathbb{S}^{n-1}}\frac{h(K_t,\xi)-h_0(\xi)}{t}d\mu_q^{tor}(K_0,\xi)\\
&\ \ \leq\frac{q-1}{q+n(q-1)}
T_q(K_0)^{\frac{q-2}{q-1}}\lim_{t\rightarrow0^+}
\int_{\mathbb{S}^{n-1}}\frac{h_t(\xi)-h_0(\xi)}{t}d\mu_q^{tor}(K_0,\xi)\\
&\ \ =\frac{q-1}{q+n(q-1)}
T_q(K_0)^{\frac{q-2}{q-1}}
\int_{\mathbb{S}^{n-1}}h_+^\prime(0,\xi)d\mu_q^{tor}(K_0,\xi)\\
&\ \ =l.
\end{align*}
This, together with (\ref{3.11}), yields
\begin{align*}
l&\geq\lim_{t\rightarrow0^+}\frac{T_q(K_0,K_t)-T_q(K_0)}{t}\\
&\geq\lim_{t\rightarrow0^+}\frac{T_q(K_0)^{1-\frac{1}
{q+n(q-1)}}T_q(K_t)^{\frac{1}{q+n(q-1)}}-T_q(K_0)}{t},
\end{align*}
and hence
\begin{align}\label{3.18}
l\geq T_q(K_0)^{1-\frac{1}{q+n(q-1)}}\lim_{t\rightarrow 0^+}
\frac{T_q(K_t)^{\frac{1}{q+n(q-1)}}
-T_q(K_0)^{\frac{1}{q+n(q-1)}}}{t}.
\end{align}
Combining (\ref{3.17}) and (\ref{3.18}), we see that
\begin{align}\label{3.19}
l=T_q(K_0)^{1-\frac{1}{q+n(q-1)}}\lim_{t\rightarrow 0^+}
\frac{T_q(K_t)^{\frac{1}{q+n(q-1)}}
-T_q(K_0)^{\frac{1}{q+n(q-1)}}}{t}.
\end{align}

From the Aleksandrov's convergence lemma and the continuity of $q$-torsional rigidity on $\mathcal{K}_o^n$, we know that $T_q: C_+(\mathbb{S}^{n-1})\rightarrow\mathbb{R}$  is continuous.
Define a function $f: I\rightarrow\mathbb{R}$ by $f(t)=T_q(K_t)^{\frac{1}{q+n(q-1)}}$. Identity (\ref{3.19})
shows that the right derivative of $f^n$ exists at $0$ and that
\begin{align*}
&\lim_{t\rightarrow 0^+}\frac{f(t)^{q+n(q-1)}-
f(0)^{q+n(q-1)}}{t}\\
&\ \ =(q+n(q-1))f(0)^{q
+n(q-1)-1}\lim_{t\rightarrow 0^+}\frac{f(t)-f(0)}{t}.
\end{align*}
Thus the definition of $f$ and (\ref{3.19}) prove that
$$\lim_{t\rightarrow 0^+}\frac{T_q(K_t)-T_q(K_0)}{t}
=(q+n(q-1))l.$$
The proof is completed.
\end{proof}

\begin{lemma}\label{lem3.8}
{\it Let $I\subset\mathbb{R}$ be an interval containing $0$ in its interior and let
$$h(t,\xi): I\times\mathbb{S}^{n-1}\rightarrow(0,\infty)$$
be continuous and such that the convergence in
\begin{align*}
h^\prime(0,\xi)=\lim_{t\rightarrow0}\frac{h(t,\xi)-h(0,\xi)}{t}
\end{align*}
is uniform on $\mathbb{S}^{n-1}$. If $\{K_t\}_{t\in I}$ is the family of Aleksandrov bodies associated with $h_t$, and $q>1$, then
\begin{align}\label{3.20}
\frac{d}{dt}T_q(K_t)\bigg|_{t=0}=(q-1)
T_q(K_0)^{\frac{q-2}{q-1}}\int_{\mathbb{S}^{n-1}}
h^\prime(0,\xi)d\mu_q^{tor}(K_0,\xi).
\end{align}}
\end{lemma}

\begin{proof}
From Lemma \ref{lem3.7} we only need to
\begin{align}\label{3.21}
\lim_{t\rightarrow0^-}\frac{T_q(K_t)-T_q(K_0)}{t}
=(q-1)
T_q(K_0)^{\frac{q-2}{q-1}}\int_{\mathbb{S}^{n-1}}
h^\prime(0,\xi)d\mu_q^{tor}(K_0,\xi).
\end{align}
To that end, defined $\widetilde{h}(t,\xi):-I\times \mathbb{S}^{n-1}\rightarrow(0,\infty)$ by $\widetilde{h}(t,\xi)=h(-t,\xi)$. For the corresponding family $\{\widetilde{K}_{-t}\}_{t\in I}$ of Aleksandrov bodies associated with $\widetilde{h}$ we have $\widetilde{K}_{-t}=K_t$ and $\widetilde{K}_0=K_0$. Thus, by Lemma \ref{lem3.7}
\begin{align*}
\lim_{t\rightarrow0^-}\frac{T_q(K_t)-T_q(K_0)}{-t}
&=\lim_{t\rightarrow0^+}\frac{T_q(\widetilde{K}_t)-T_q(\widetilde{K}_0)}{t}\\
&=(q-1)
T_q(K_0)^{\frac{q-2}{q-1}}\int_{\mathbb{S}^{n-1}}
\widetilde{h}^\prime(0,\xi)d\mu_q^{tor}(K_0,\xi).
\end{align*}
Obviously, $\widetilde{h}^\prime(0,\xi)=-h^\prime(0,\xi)$, which immediately implies (\ref{3.21}).
\end{proof}

\section{ The $L_p$ $q$-torsional measure }

In this part, we define the $L_p$ $q$-torsional measure by establishing $L_p$ variational formula for $q$-torsional rigidity of general convex bodies without smoothness condition.

\begin{theorem}\label{the4.1}
Let $K\in\mathcal{K}_o^n$, $1\leq p<\infty$ and $q>1$. If $L$ is a compact convex set containing the origin, then
$$\frac{d}{dt}T_q(K+_pt\cdot L)\bigg|_{t=0}=\frac{q-1}{p}
T_q(K)^{\frac{q-2}{q-1}}
\int_{\mathbb{S}^{n-1}}h(L,\xi)^ph(K,\xi)^{1-p}d\mu_q^{tor}(K,\xi).$$
\end{theorem}

\begin{proof}
For $K,L\in\mathcal{K}_o^n$ and $1\leq p<\infty$. By the definition of Firey's $p$-sum, then we have
\begin{align*}
\lim_{t\rightarrow0}\frac{h(K+_pt\cdot L,\cdot)-h(K,\cdot)}{t}
&=\lim_{t\rightarrow0}\frac{(h(K,\cdot)^p+th(L,\cdot)^p)^{\frac{1}{p}}
-h(K,\cdot)}{t}\\
&=\frac{1}{p}(h(K,\cdot)^p+th(L,\cdot)^p)^{\frac{1}{p}-1}\bigg|_{t=0}h(L,\cdot)^p\\
&=\frac{h(L,\cdot)^ph(K,\cdot)^{1-p}}{p}.
\end{align*}
Combining Lemma \ref{lem3.8}, we have
$$\frac{d}{dt}T_q(h_K+_pt\cdot L)\bigg|_{t=0}=\frac{q-1}{p}
T_q(K)^{\frac{q-2}{q-1}}
\int_{\mathbb{S}^{n-1}}h(L,\xi)^ph(K,\xi)^{1-p}d\mu_q^{tor}(K,\xi).$$
This completes the proof of Theorem \ref{the4.1}.
\end{proof}

We now give the following definition for the new geometric measure produced by the variational formula in Theorem \ref{the4.1}.

\begin{definition}\label{def4.3}
{\it Suppose $p\in\mathbb{R}$ and $q>1$. For $K\in\mathcal{K}_o^n$, the finite Borel measure $\mu_{p,q}^{tor}(K,\cdot)$ defined, for each Borel set $\eta\subseteq\mathbb{S}^{n-1}$, by
$$\mu_{p,q}^{tor}(K,\eta)=\int_\eta
h(K,\cdot)^{1-p}d\mu_q^{tor}(K,\cdot),$$
is called the $L_p$ $q$-torsional measure.}
\end{definition}

\begin{definition}\label{def4.4}
{\it Suppose $p\in\mathbb{R}$ and $q>1$. For $K,L\in\mathcal{K}_o^n$, if $p\neq0$, define
\begin{align}\label{4.1}
T_{p,q}(K,L)=\frac{q-1}{q+n(q-1)}
T_q(K)^{\frac{q-2}{q-1}}
\int_{\mathbb{S}^{n-1}}h(L,\xi)^pd\mu_{p,q}^{tor}(K,\xi),
\end{align}
and call it the $L_p$ mixed $q$-torsional rigidity of $(K,L)$.
Obviously, $T_{1,q}(K,L)=T_q(K,L)$ and $T_{p,q}(K,K)=T_q(K)$.}
\end{definition}

From the Definition \ref{def4.3}, together with the positive homogeneity and weak convergence of $\mu_q^{tor}(K,\cdot)$, we obtain the following result.

\begin{lemma}\label{lem4.5}
 For $K, K_i\in\mathcal{K}_o^n$ and $i\in\mathbb{N}$.

(1) Let $p\in\mathbb{R}$ and $q>1$, then $\mu_{p,q}^{tor}(sK,\cdot)=
s^{n(q-1)+q-p}\mu_{p,q}^{tor}(K,\cdot)$ for $s>0$.

(2) Let $p\in\mathbb{R}$ and $q>1$. If $K_i\rightarrow K$ in the Hausdorff metric, then $\mu_{p,q}^{tor}(K_i,\cdot)\rightarrow\mu_{p,q}^{tor}(K,\cdot)$ weakly, as $i\rightarrow\infty$.
\end{lemma}

Next, we establish the natural $L_p$ extension
of the Brunn-Minkowski and Minkowski type inequalities
for $q$-torsional rigidity.

\begin{theorem}\label{th4.6}
{\it For $K,L\in\mathcal{K}_o^n$. If $1<p<\infty$ and $q>1$, then
\begin{align}\label{4.2}
T_q(K+_pL)^{\frac{p}{n(q-1)+q}}\geq T_q(K)^{\frac{p}{n(q-1)+q}}
+T_q(L)^{\frac{p}{n(q-1)+q}},
\end{align}
with equality if and only if $K$ and $L$ are dilates.}
\end{theorem}

\begin{theorem}\label{th4.7}
{\it For $K,L\in\mathcal{K}_o^n$. If $1<p<\infty$ and $q>1$, then
\begin{align}\label{4.3}
T_{p,q}(K,L)^{n(q-1)+q}\geq T_q(K)^{n(q-1)+q-p}
T_q(L)^{p},
\end{align}
with equality if and only if $K$ and $L$ are dilates.}
\end{theorem}

For $p>1$, $q>1$, and $K,L\in\mathcal{K}_o^n$ of class $C_+^2$, the Theorems \ref{th4.6} and \ref{th4.7} were proved in \cite{SX}. When $K,L$ are arbitrary convex body in $\mathcal{K}_o^n$, the proof of the Theorems \ref{th4.6} and \ref{th4.7} are  very similar and thus omitted.\\

\section{ The $L_p$ Minkowski problem of $q$-torsional rigidity for discrete measure}

In this part, we prove the existence of solution to $L_p$ Minkowski problem of $q$-torsional rigidity for discrete measure when $0<p<1$ and $q>1$.
First, we study an extremal problem under translation transforms.
Next, we establish the relationship between extremal problem and $L_p$ Minkowski problem. Finally, we give the solution for discrete measure.

Let $\mathcal{P}$ be the set of polytopes in $\mathbb{R}^n$.
Suppose the unit vectors $\xi_1,...,\xi_N$ $(N\geq n+1)$ are not concentrated on any closed hemisphere of $\mathbb{S}^{n-1}$. Let $\mathcal{P}(\xi_1,...,\xi_N)$ be the set with $P\in \mathcal{P}(\xi_1,...,\xi_N)$ such that for
fixed $a_1,...,a_N\geq0$,
$$P=\bigcap_{k=1}^N\{x\in\mathbb{R}^n: x\cdot\xi_k\leq a_k\}.$$
Obviously, for $P\in \mathcal{P}(\xi_1,...,\xi_N)$, $P$ has at most $N$ facets,
and the outer unit normals of $P$ are a subset of $\{\xi_1,...,\xi_N\}$. Let $\mathcal{P}_N(\xi_1,...,\xi_N)$ be the
subset of $\mathcal{P}(\xi_1,...,\xi_N)$ such that a polytope $P\in \mathcal{P}_N(\xi_1,...,\xi_N)$, if $P\in \mathcal{P}(\xi_1,...,\xi_N)$ and $P$ has exactly $N$ facets.

Suppose $c_1,...,c_N$ are positive real numbers and the unit vectors $\xi_1,...,\xi_N$ are not concentrated on any closed hemisphere of $\mathbb{S}^{n-1}$. Let
$$\mu=\sum_{k=1}^Nc_k\delta_{\xi_k}(\cdot)$$
be the discrete measure on $\mathbb{S}^{n-1}$,
where $\delta$ is the Kronecker delta.
More knowledge about polytopes and discrete measures can be refer to \cite{G1}.\\

\subsection{An extremal problem}
\vglue-10pt
 \indent

For $\Omega\in \mathcal{P}(\xi_1,...,\xi_N)$ and $0<p<1$, we define the functional $\Phi(\Omega,\cdot): \Omega\rightarrow\mathbb{R}$ as
\begin{align}\label{5.1.1}
\Phi(\Omega,x)=\sum_{k=1}^Nc_k(h(\Omega,\xi_k)-x\cdot\xi_k)^p.
\end{align}
We will show that there is a unique point $x_\Omega\in$ Int$(\Omega)$ such that $\Phi(\Omega,x)$ attains the maximum.

\begin{lemma}\label{lem5.1.1}
Suppose $0<p<1$, the unit vectors $\xi_1,...,\xi_N$ are not concentrated on any closed hemisphere of $\mathbb{S}^{n-1}$ and
$\Omega\in \mathcal{P}(\xi_1,...,\xi_N)$. Then there exists a unique point $x_\Omega\in$ Int$(\Omega)$ such that
$$\Phi(\Omega,x_\Omega)=\max_{x\in\Omega}\Phi(\Omega,x).$$
\end{lemma}

\begin{proof}
Firstly, we prove the uniqueness of the maximal point. Assume $x_1,x_2\in$ Int$(\Omega)$ and
$$\Phi(\Omega,x_1)=\Phi(\Omega,x_2)=\max_{x\in\Omega}\Phi(\Omega,x).$$
From (\ref{5.1.1}), and using the Jensen inequality, we get
\begin{align*}
\Phi(\Omega,\frac{1}{2}(x_1+x_2))&=\sum_{k=1}^Nc_k(h(\Omega,\xi_k)
-\frac{1}{2}(x_1+x_2)\cdot\xi_k)^p\\
&=\sum_{k=1}^Nc_k(\frac{1}{2}(h(\Omega,\xi_k)-x_1\cdot\xi_k)
+\frac{1}{2}(h(\Omega,\xi_k)-x_2\cdot\xi_k))^p\\
&\geq\frac{1}{2}\sum_{k=1}^Nc_k(h(\Omega,\xi_k)-x_1\cdot\xi_k)^p
+\frac{1}{2}\sum_{k=1}^Nc_k(h(\Omega,\xi_k)-x_2\cdot\xi_k)^p\\
&=\frac{1}{2}\Phi(\Omega,x_1)+\frac{1}{2}\Phi(\Omega,x_2)\\
&=\max_{x\in\Omega}\Phi(\Omega,x).
\end{align*}
Since $\Omega$ is convex, $\frac{1}{2}(x_1+x_2)\in\Omega$, then the above equality holds. The equality condition of Jensen inequality means that
$$h(\Omega,\xi_k)-x_1\cdot\xi_k=h(\Omega,\xi_k)-x_2\cdot\xi_k, \ \ k=1,...,N,$$
that is
$$x_1\cdot\xi_k=x_2\cdot\xi_k, \ \ k=1,...,N.$$
Since the unit vector $\xi_1,...,\xi_N$ are not concentrated on any closed hemisphere, it follows that $x_1=x_2$. Thus the uniqueness is proved.

Next, we prove the existence of the maximal point. Since $\Phi(\Omega,x)$ is continuous in $x\in\Omega$ and $\Omega$ is compact, then $\Phi(\Omega,x)$ attains its maximum at a point of $\Omega$, denoted by $x_\Omega$. Thus we only need to prove $x_\Omega\in$ Int$(\Omega)$. We use proof by contradiction. Suppose $x_\Omega\in\partial\Omega$ with
$$h(\Omega,\xi_k)-x_\Omega\cdot\xi_k=0$$
for $k=\{i_1,...,i_m\}$, and
$$h(\Omega,\xi_k)-x_\Omega\cdot\xi_k>0$$
for $k=\{1,...,N\}\backslash\{i_1,...,i_m\}$, where $1\leq i_1<\cdot\cdot\cdot<i_m\leq N$ and $1\leq m\leq N-1$.
Fix $y_0\in$ Int$(\Omega)$, let $\xi_0=\frac{y_0-x_\Omega}{|y_0-x_\Omega|}$. Then for sufficiently small $\varepsilon>0$, it follows that $x_\Omega+\varepsilon\xi_0\in$ Int$(\Omega)$. In the following, we aim to show that
$\Phi(\Omega,x_\Omega+\varepsilon\xi_0)-\Phi(\Omega,x_\Omega)>0$, which will contradict the maximality of $\Phi$ at $x_\Omega$. Consequently, $x_\Omega\in$ Int$(\Omega)$.

Let
\begin{align}\label{5.1.2}
[h(\Omega,\xi_k)-(x_\Omega+\varepsilon\xi_0)\cdot\xi_k]-
[h(\Omega,\xi_k)-x_\Omega\cdot\xi_k]=-(\xi_0\cdot\xi_k)
\varepsilon=\alpha_k\varepsilon,
\end{align}
where $\alpha_k=-(\xi_0\cdot\xi_k)$. Since $h(\Omega,\xi_k)-x_\Omega\cdot\xi_k=0$
for $k\in\{i_1,...,i_m\}$ and $y_0$ is an interior point of $\Omega$, $\alpha_k>0$ for $k\in\{i_1,...,i_m\}$. Let
\begin{align}\label{5.1.3}
\alpha_0=\min\{h(\Omega,\xi_k)-x_\Omega\cdot\xi_k: k=\{1,...,N\}\backslash\{i_1,...,i_m\}\}>0,
\end{align}
and choose $\varepsilon>0$ small enough such that $x_\Omega+\varepsilon\xi_0\in$ Int$(\Omega)$ and
\begin{align}\label{5.1.4}
\min\{h(\Omega,\xi_k)-(x_\Omega+\varepsilon\xi_0)\cdot\xi_k: k=\{1,...,N\}\backslash\{i_1,...,i_m\}\}>\frac{\alpha_0}{2}.
\end{align}

Obviously, for $0<p<1$ and $y_0, y_0+\triangle y\in(\frac{\alpha_0}{2},+\infty)$,
$$|(y_0+\triangle y)^p-y_0^p|<p\bigg(\frac{\alpha_0}{2}\bigg)^{p-1}|\triangle y|.$$
From this, the fact that $h(\Omega,\xi_k)=x_\Omega\cdot\xi_k$, $\alpha_k>0$ for $k\in\{i_1,...,i_m\}$, (\ref{5.1.2}), (\ref{5.1.3}) and (\ref{5.1.4}), it follows that
\begin{align*}
&\Phi(\Omega,x_\Omega+\varepsilon\xi_0)-\Phi(\Omega,x_\Omega)\\
&=\sum_{k=1}^Nc_k[(h(\Omega,\xi_k)-(x_\Omega+\varepsilon\xi_0)\cdot\xi_k)^p
-(h(\Omega,\xi_k)-x_\Omega\cdot\xi_k)^p]\\
&\geq\sum_{k\in\{i_1,...,i_m\}}c_k(\alpha_k\varepsilon)^p
-\sum_{k\in\{1,...,N\}\backslash\{i_1,...,i_m\}}c_k
\bigg|(h(\Omega,\xi_k)-x_\Omega\cdot\xi_k+\alpha_k\varepsilon)^p\\
&\ \ \ \ -(h(\Omega,\xi_k)-x_\Omega\cdot\xi_k)^p\bigg|\\
&\geq\bigg(\sum_{k\in\{i_1,...,i_m\}}c_k\alpha_k^p\bigg)\varepsilon^p
-\sum_{k\in\{1,...,N\}\backslash\{i_1,...,i_m\}}
c_kp\bigg(\frac{\alpha_0}{2}\bigg)^{p-1}|\alpha_k\varepsilon|\\
&=\bigg(\sum_{k\in\{i_1,...,i_m\}}c_k\alpha_k^p
-\sum_{k\in\{1,...,N\}\backslash\{i_1,...,i_m\}}
c_kp\bigg(\frac{\alpha_0}{2}\bigg)^{p-1}|
\alpha_k|\varepsilon^{1-p}\bigg)\varepsilon^p.
\end{align*}
Thus, there exists a small enough $\varepsilon_0>0$ such that
$x_\Omega+\varepsilon_0\xi_0\in$ Int$(\Omega)$ and
$$\Phi(\Omega,x_\Omega+\varepsilon_0\xi_0)
-\Phi(\Omega,x_\Omega)>0.$$
This contradicts the definition of $x_\Omega$. Therefore $x_\Omega\in$ Int$(\Omega)$.
\end{proof}

\begin{lemma}\label{lem5.1.2}
Let $x_\Omega, x_{\Omega_i}$ be the maximal point of the functional $\Phi$ on $\Omega, \Omega_i\in \mathcal{P}(\xi_1,...,\xi_N)$. Suppose $\Omega_i\rightarrow \Omega$ as $i\rightarrow\infty$, then
$x_{\Omega_i}\rightarrow x_\Omega$ and  $\Phi(\Omega_i,x_{\Omega_i})\rightarrow \Phi(\Omega,x_{\Omega})$ as $i\rightarrow\infty$.
\end{lemma}

\begin{proof}
Since $\Omega_i\rightarrow \Omega$ as $i\rightarrow\infty$, we have
$$x_{\Omega_i}\in\Omega_i\subseteq \Omega+B.$$
This implies $\{x_{\Omega_i}\}_i$ is a bounded sequence.
Let $\{x_{\Omega_{i_j}}\}_j$ be a convergent subsequence of $\{x_{\Omega_i}\}_i$.

Assume $\{x_{\Omega_{i_j}}\}_j\rightarrow x^\prime$ and $x^\prime\neq x_\Omega$.
By the Theorem 1.8.8 in \cite{S}, it follows that $x^\prime\in\Omega$. Hence
$$\Phi(\Omega,x^\prime)<\Phi(\Omega,x_\Omega).$$
From the continuity of $\Phi(\Omega,x)$ in $\Omega$ and $x$, we have
$$\lim_{j\rightarrow\infty}\Phi(\Omega_{i_j},x_{\Omega_{i_j}})=
\Phi(\Omega,x^\prime).$$

Meanwhile, by the Theorem 1.8.8 in \cite{S}, for $x_\Omega\in\Omega$, there exists a $y_{i_j}\in\Omega_{i_j}$ such that $y_{i_j}\rightarrow x_\Omega$. Then we have
$$\lim_{j\rightarrow\infty}\Phi(\Omega_{i_j},y_{{i_j}})=
\Phi(\Omega,x_\Omega).$$
Hence
\begin{align}\label{5.1.5}
\lim_{j\rightarrow\infty}\Phi(\Omega_{i_j},x_{\Omega_{i_j}})<
\lim_{j\rightarrow\infty}\Phi(\Omega_{i_j},y_{{i_j}}).
\end{align}
However, for any $\Omega_{i_j}$,
$$\Phi(\Omega_{i_j},x_{\Omega_{i_j}})\geq \Phi(\Omega_{i_j},y_{{i_j}}),$$
then we have
$$\lim_{j\rightarrow\infty}\Phi(\Omega_{i_j},x_{\Omega_{i_j}})\geq \lim_{j\rightarrow\infty}\Phi(\Omega_{i_j},y_{{i_j}}),$$
which contradicts (\ref{5.1.5}). Thus, $x_{\Omega_{i_j}}\rightarrow x_\Omega$, and
$x_{\Omega_i}\rightarrow x_\Omega$. From the continuity of $\Phi$, then
$$\Phi(\Omega_i,x_{\Omega_i})\rightarrow\Phi(\Omega,x_\Omega).$$
This completes the proof.
\end{proof}

\begin{lemma}\label{lem5.1.3}
Suppose $\Omega\in \mathcal{P}(\xi_1,...,\xi_N)$, then

(1) $\Phi(\Omega+y,x_{\Omega+y})=\Phi(\Omega,x_\Omega)$,\ \ for $y\in\mathbb{R}^n$.

(2) $\Phi(\lambda\Omega,x_{\lambda\Omega})=\lambda^p\Phi(\Omega,x_\Omega)$,\ \ for $\lambda>0$.

\end{lemma}

\begin{proof}
From (\ref{5.1.1}), we have
\begin{align*}
\Phi(\Omega+y,x_{\Omega+y})&=\max_{z\in\Omega+y}\Phi(\Omega+y,z)\\
&=\max_{z-y\in\Omega}\sum_{k=1}^Nc_k(h({\Omega+y},\xi_k)-z\cdot\xi_k)^p\\
&=\max_{z-y\in\Omega}\sum_{k=1}^Nc_k(h({\Omega},\xi_k)-(z-y)\cdot\xi_k)^p\\
&=\max_{x\in\Omega}\sum_{k=1}^Nc_k(h({\Omega},\xi_k)-x\cdot\xi_k)^p\\
&=\Phi(\Omega,x_\Omega).
\end{align*}
In the same way, we can get a proof of (2).
\end{proof}

\
\subsection{An extremal problem and the $L_p$ Minkowski problem}
\vglue-10pt
 \indent

Suppose $0<p<1$ and $q>1$, in the following, we study the extremal problem
\begin{align}\label{5.2.1}
\inf\{\max_{x\in\Omega}\Phi(\Omega,x): \Omega\in \mathcal{P}(\xi_1,...,\xi_N), \ T_q(\Omega)=1\},
\end{align}
and show that its solution is exactly the solution of the $L_p$ Minkowski problem for $q$-torsional rigidity we are concerned with.

\begin{lemma}\label{lem5.2.1}
Suppose $P\in\mathcal{P}(\xi_1,...,\xi_N)$ with normal vector $\xi_1,...,\xi_N$. If $P$ is the solution to problem (\ref{5.2.1}), and $x_P=o$, then
$$\lambda h(P,\cdot)^{1-p}d\mu_q^{tor}(P,\cdot)=d\mu,$$
where $\lambda=\frac{q-1}{q+n(q-1)}
\sum_{k=1}^Nc_kh(P,\xi_k)^p$.
\end{lemma}

\begin{proof}
For $\delta_1,...,\delta_N>0$ and sufficiently small $|t|>0$. Let
$$P_t=\{x: x\cdot\xi_k\leq h(P,\xi_k)+t\delta_k, \ k=1,...,N\}$$
and
$$\gamma(t)P_t
=T_q(P_t)^{-\frac{1}{q+n(q-1)}}P_t.$$
Then $T_q(\gamma(t)P_t)=1$, $\gamma(t)P_t\in \mathcal{P}_N(\xi_1,...,\xi_N)$ and $\gamma(t)P_t\rightarrow P$ as $t\rightarrow0$.

We denote by $x(t)=x_{\gamma(t)P_t}$. Let
\begin{align}\label{5.2.01}
\nonumber \Phi(\gamma(t)P_t,x(t))&=\max_{x\in\gamma(t)P_t}\sum_{k=1}^Nc_k(\gamma(t)
h({P_t},\xi_k)-x\cdot\xi_k)^p\\
&=\sum_{k=1}^Nc_k(\gamma(t)
h({P_t},\xi_k)-x(t)\cdot\xi_k)^p.
\end{align}
Since $x(t)$ is an interior point of $\gamma(t)P_t$, by (\ref{5.2.01}), we have
$$\sum_{k=1}^Nc_k\frac{\xi_{k,i}}{[\gamma(t)h({P_t},\xi_k)-x(t)\cdot\xi_k]^{1-p}}=0,\ i=1,...,n,$$
where $\xi_k=(\xi_{k,1},...,\xi_{k,n})^\top$.
Let $t=0$, then $P_0=P$, $\gamma(0)=1$, $x(0)=o$ and
\begin{align}\label{5.2.2}
\sum_{k=1}^Nc_k\frac{\xi_{k,i}}{h(P,\xi_k)^{1-p}}=0, \ i=1,...,n.
\end{align}
Therefore
\begin{align}\label{5.2.02}
\sum_{k=1}^Nc_k\frac{\xi_{k}}{h(P,\xi_k)^{1-p}}=0.
\end{align}

Now we need to show $x^\prime(t)\bigg|_{t=0}$ exists. Let
$$y_i(t,x_1,...,x_n)=\sum_{k=1}^Nc_k\frac{\xi_{k,i}}{[\gamma(t)h({P_t},\xi_k)-
(x_1\xi_{k,1}+\cdot\cdot\cdot+x_n\xi_{k,n})]^{1-p}}$$
for $i=1,...,n$. Then
$$\frac{\partial y_i}{\partial x_j}\bigg|_{0,...,0}
=\sum_{k=1}^N\frac{(1-p)c_k}{h(P,\xi_k)^{2-p}}\xi_{k,i}\xi_{k,j}.$$
Thus
$$\bigg(\frac{\partial y}{\partial x}\bigg|_{0,...,0}\bigg)_{n\times n}=\sum_{k=1}^N\frac{(1-p)c_k}{h(P,\xi_k)^{2-p}}\xi_{k}\xi_{k}^\top.$$
For $x\in\mathbb{R}^n$ with $x\neq0$. Since $\xi_1,...,\xi_N$ are not concentrated on any closed hemisphere, then there exists a $\xi_{i_0}\in\{\xi_1,...,\xi_N\}$ such that $\xi_{i_0}\cdot x\neq0$.
Thus
\begin{align*}
&x^\top\bigg(\sum_{k=1}^N\frac{(1-p)c_k}{h(P,\xi_k)^{2-p}}
\xi_{k}\xi_{k}^\top\bigg)x\\
&=\sum_{k=1}^N\frac{(1-p)c_k}{h(P,\xi_k)^{2-p}}(x\cdot\xi_{k})^2\\
&\geq\frac{(1-p)c_{i_0}}{h(P,\xi_{i_0})^{2-p}}(x\cdot\xi_{i_0})^2\\
&>0,
\end{align*}
which implies that $\bigg(\frac{\partial y}{\partial x}\bigg|_{0,...,0}\bigg)_{n\times n}$ is positively define.
This, combining (\ref{5.2.2}) and
the inverse function theorem, it follows that $x^\prime(0)=(x_1^\prime(0),...,x_n^\prime(0))$ exists.

Next, we can finish the proof. Since the functional $\Phi$ attains its minimum at the polytope $P$, From (\ref{3.20}) and (\ref{5.2.02}), we have
\begin{align*}
0&=\frac{1}{p}\frac{d\Phi(\gamma(t)P_t,x(t))}{dt}\bigg|_{t=0}\\
&=\sum_{j=1}^Nc_jh(P,\xi_j)^{p-1}\bigg[h(P,\xi_j)
\bigg(-\frac{1}{q+n(q-1)}\bigg)
\frac{dT_q(P_t)}{dt}\bigg|_{t=0}
+\delta_j-x^\prime(0)\cdot\xi_j\bigg]\\
&=\sum_{j=1}^Nc_jh(P,\xi_j)^{p-1}\bigg[h(P,\xi_j)
\bigg(-\frac{q-1}{q+n(q-1)}\bigg)
\bigg(\sum_{k=1}^N\delta_k\mu_q^{tor}(P,\{\xi_k\})\bigg)
+\delta_j\bigg]\\
&\ \ -x^\prime(0)\bigg(\sum_{j=1}^Nc_jh(P,\xi_j)^{p-1}\xi_j\bigg)\\
&=\sum_{j=1}^Nc_jh(P,\xi_j)^{p-1}\bigg[h(P,\xi_j)
\bigg(-\frac{q-1}{q+n(q-1)}\bigg)
\bigg(\sum_{k=1}^N\delta_k\mu_q^{tor}(P,\{\xi_k\})\bigg)
+\delta_j\bigg]\\
&=\sum_{j=1}^N\delta_j\bigg[c_jh(P,\xi_j)^{p-1}-
\bigg(\frac{q-1}{q+n(q-1)}\bigg)
\bigg(\sum_{k=1}^Nc_kh(P,\xi_k)^p\bigg)
\mu_q^{tor}(P,\{\xi_k\})\bigg].
\end{align*}
For arbitrary positive real numbers $\delta_1,...,\delta_N$, we have
$$\frac{q-1}{q+n(q-1)}
\bigg(\sum_{k=1}^Nc_kh(P,\xi_k)^p\bigg)
\mu_q^{tor}(P,\{\xi_k\})=c_jh(P,\xi_j)^{p-1}$$
for $j=1,...,N$. Since $P$ is containing the origin $o$ in its interior, then $h(P,\xi_j)>0$, and thus
$$\frac{q-1}{q+n(q-1)}
\bigg(\sum_{k=1}^Nc_kh(P,\xi_k)^p\bigg)h(P,\xi_j)^{1-p}
\mu_q^{tor}(P,\{\xi_k\})=c_j$$
for $j=1,...,N$. Therefore, Therefore
$$\lambda h(P,\cdot)^{1-p}d\mu_q^{tor}(P,\cdot)=d\mu,$$
where $\lambda=\frac{q-1}{q+n(q-1)}\sum_{k=1}^Nc_kh(P,\xi_k)^p$.
\end{proof}

The following lemma shows that the solution to (\ref{5.2.1}) is just a scaling of the solution to the $L_p$ Minkowski problem for $q$-torsional rigidity.

\begin{lemma}\label{lem5.2.2}
Suppose $P\in\mathcal{P}(\xi_1,...,\xi_N)$ with normal vector $\xi_1,...,\xi_N$. If $P$ is the solution to problem (\ref{5.2.1}), and $x_P=o$.  Then for
$$\lambda_0=\bigg(\frac{q-1}{q
+n(q-1)}\sum_{k=1}^Nc_kh(P,\xi_k)^p\bigg)^{\frac{1}{q
+n(q-1)-p}},$$
we have
$$d\mu_{p,q}^{tor}(\lambda_0P,\cdot)=d\mu.$$
\end{lemma}

\begin{proof}
Let $s>0$ and $P\in \mathcal{P}(\xi_1,...,\xi_N)$. Then
\begin{align}\label{5.2.3}
d\mu_{p,q}^{tor}(sP,\cdot)=s^{{q
+n(q-1)}-p}h(P,\cdot)^{1-p}d\mu_{q}(P,\cdot)
=s^{{q
+n(q-1)}-p}d\mu_{p,q}^{tor}(P,\cdot)
\end{align}
Since $0<p<1$ and $q>1$, then $n(q-1)\neq p-q$. If $P$ is the solution to (\ref{5.2.1}), By Lemma \ref{lem5.2.1}, we have
$$\lambda d\mu_{p,q}^{tor}(P,\cdot)
=\lambda h(P,\cdot)^{1-p}d\mu_q(P,\cdot)=d\mu,$$
where $\lambda=\frac{q-1}{q
+n(q-1)}\sum_{k=1}^Nc_kh(P,\xi_k)^p$.  This together with (\ref{5.2.3}), we have
$$d\mu_{p,q}^{tor}(\lambda_0P,\cdot)=d\mu,$$
where $\lambda_0=\lambda^{\frac{1}{q+n(q-1)-p}}$.  This completes the proof.
\end{proof}
\

\subsection{Existence of solutions to the $L_p$ Minkowski problem for $q$-torsional rigidity}
\vglue-10pt
 \indent

We also need the following two lemmas to complete the existence of solution to $L_p$ Minkowski problem of $q$-torsional rigidity for discrete measure when $0<p<1$ and $q>1$.

\begin{lemma}\label{lem5.3.1}
Suppose $P\in\mathcal{P}(\xi_1,...,\xi_N)$ with normal vector $\xi_1,...,\xi_N$, and $0<p<1$. If $P$ is the solution to problem (\ref{5.2.1}), and $x_P=o$.
Then $P$ has exactly $N$ facets whose normal vectors are $\xi_1,...,\xi_N$.
\end{lemma}

\begin{proof}
We argue be contradiction. Assume that $\xi_{i_0}\in\{\xi_1,...,\xi_N\}$, but the support set $\mathcal{F}(P,\xi_{i_0})=P\cap H(P,\xi_{i_0})$ is not a facet of $P$.

Fix $\delta>0$, let
$$P_\delta=P\cap\{x:x\cdot\xi_{i_0}\leq h(P,\xi_{i_0})-\delta\}\in\mathcal{P}(\xi_1,...,\xi_N)$$
and
$$\tau P_\delta=\tau(\delta)P_\delta
=T_q(P_\delta)^{-\frac{1}{q+n(q-1)}}P_\delta.$$
Then $T_q(\tau P_\delta)=1$ and $\tau P_\delta\rightarrow P$ as $\delta\rightarrow0^+$. By the Lemma \ref{5.1.2}, we see that $x_{P_\delta}\rightarrow x_P=o\in$ int$(P)$ as $\delta\rightarrow0^+$.
Thus, for sufficiently small $\delta>0$, we can assume that $x_{P_\delta}\in$ Int$(P)$ and
$$h(P,\xi_k)-x_{P_\delta}\cdot\xi_k>\delta>0,\ \ k=1,...,N.$$

In the following, we show $\Phi(\tau P_\delta,x_{\tau P_\delta})<\Phi(P,o)$, which contradicts the fact that $\Phi(P,o)$ is the minimum. Since
\begin{align*}
\Phi(\tau P_\delta,x_{\tau P_\delta})&=\tau^p\sum_{k=1}^Nc_k(h({P_\delta},\xi_k)-x_{P_\delta}\cdot\xi_k)^p\\
&=\tau^p\bigg(\sum_{k=1}^Nc_k(h({P},\xi_k)-x_{P_\delta}\cdot\xi_k)^p\bigg)
+\tau^pc_{i_0}(h({P},\xi_{{i_0}})
-x_{P_\delta}\cdot\xi_{i_0}-\delta)^p\\
& \ \ -\tau^pc_{i_0}(h({P},\xi_{i_0})-x_{P_\delta}\cdot\xi_{i_0})^p\\
&=\Phi(P,x_{P_\delta})+G(\delta),
\end{align*}
where
\begin{align*}
G(\delta)&=(\tau^p-1)\bigg(\sum_{k=1}^N
c_k(h({P},\xi_k)-x_{P_\delta}\cdot\xi_k)^p\bigg)\\
&\ \ +c_{i_0}\tau^p[(h(P,\xi_{i_0})-x_{P_\delta}\cdot\xi_{i_0}-\delta)^p
-(h(P,\xi_{i_0})-x_{P_\delta}\cdot\xi_{i_0})^p].
\end{align*}
If we can prove $G(\delta)<0$, then $\Phi(\tau P_\delta,x_{\tau P_\delta})<\Phi(P,x_{P_\delta})\leq\Phi(P,o)$, as desired.

Since $0<h(P,\xi_{i_0})-x_{P_\delta}\cdot\xi_{i_0}
-\delta<h(P,\xi_{i_0})-x_{P_\delta}\cdot\xi_{i_0}<d_0$,
where $d_0$ is the diameter of $P$, by the concavity of $t^p$ on $[0,\infty)$ for $0<p<1$, it follows that
$$(h(P,\xi_{i_0})-x_{P_\delta}\cdot\xi_{i_0}-\delta)^p
-(h(P,\xi_{i_0})-x_{P_\delta}\cdot\xi_{i_0})^p<(d_0-\delta)^p-d_0^p.$$
Hence
\begin{align*}
G(\delta)&<(\tau^p-1)\bigg(\sum_{k=1}^Nc_k(h(P,\xi_k)
-x_{P_\delta}\cdot\xi_k)^p\bigg)
+c_{i_0}\tau^p((d_0-\delta)^p-d_0^p)\\
&=\tau^p((d_0-\delta)^p-d_0^p)\bigg(c_{i_0}+
\frac{1}{\tau^p}\frac{\tau^p-1}{(d_0-\delta)^p-d_0^p}
\sum_{k=1}^Nc_k(h(P,\xi_k)-x_{P_\delta}\cdot\xi_k)^p\bigg).
\end{align*}
From the formula (\ref{3.20}) and $T_q(P)=1$, we have
\begin{align*}
\lim_{\delta\rightarrow0^+}\frac{\tau^p-1}{(d_0-\delta)^p-d_0^p}
&=\lim_{\delta\rightarrow0^+}\frac{T_q(P_\delta)
^{-\frac{p}{q+n(q-1)}}-1}{(d_0-\delta)^p-d_0^p}\\
&=\frac{-\frac{p(q-1)}{q+n(q-1)}
\sum_{k=1}^N\mu_q^{tor}(P,\{\xi_k\})h^\prime(\xi_k,0)}
{-pd_0^{p-1}}\\
&=\frac{(q-1)}{q+n(q-1)}\frac{
\sum_{k=1}^N\mu_q^{tor}(P,\{\xi_k\})h^\prime(\xi_k,0)}
{d_0^{p-1}},
\end{align*}
where $h^\prime(\xi_k,0)=\lim_{\delta\rightarrow0^+}
\frac{h({P_\delta},\xi_k)-h(P,\xi_k)}{\delta}$.

Assume $\mu_q^{tor}(P,\{\xi_k\})\neq0$ for some $k$. Since $\mu_q^{tor}(P,\cdot)$ is absolutely continuous w.r.t. surface area measure $S(P,\cdot)$, then $P$ has a facet with normal vector $\xi_k$.
By the definition of $P_\delta$, we have $h({P_\delta},\xi_k)=h({P},\xi_k)$ for sufficiently small $\delta>0$. Thus $h^\prime(\xi_k,0)=0$ and
$$\sum_{k=1}^N\mu_q^{tor}
(P,\{\xi_k\})h^\prime(\xi_k,0)=0.$$
Therefore,
$$\lim_{\delta\rightarrow0^+}\frac{\tau^p-1}{(d_0-\delta)^p-d_0^p}=0.$$
This together with $(d_0-\delta)^p-d_0^p<0$, $c_{i_0}>0$ and
$$\frac{1}{\tau^p}\sum_{k=1}^Nc_k(h(P,\xi_k)-x_{P_\delta}\cdot\xi_k)
\rightarrow\sum_{k=1}^Nc_kh(P,\xi_k)^p>0$$
as $\delta\rightarrow 0^+$, then for sufficiently small $\delta>0$, $G(\delta)<0$.

Consequently, $P$ has exactly $N$ facets. This completes the proof.
\end{proof}

\begin{lemma}\label{lem5.3.2}
Suppose $\mu$ be a finite positive Borel measure on $\mathbb{S}^{n-1}$  which is not concentrated on any closed hemisphere.
Then, for $\mu=\sum_{k=1}^Nc_k\delta_{\xi_k}$, there exists a polytope $P$ solving the problem (\ref{5.2.1}).
\end{lemma}

\begin{proof}
Let
$$\beta=\inf\{\max_{x\in\Omega}\Phi(\Omega,x): \Omega\in \mathcal{P}(\xi_1,...,\xi_N),\ T_q(\Omega)=1\}.$$
Take a minimizing sequence $\{P_i\}_i$ such that $P_i\in \mathcal{P}(\xi_1,...,\xi_N)$, $x_{P_i}=o$, $T_q(P_i)=1$ and $\lim_{i\rightarrow\infty}\Phi(P_i,o)=\beta$.

Next, we prove that $\{P_i\}_i$ is bounded. Since $x_{P}=o$, by the definition of $\Phi$, it follows that
\begin{align*}
\sum_{k=1}^Nc_kh({P_i},\xi_k)^p&=\max_{x\in P_i}\sum_{k=1}^Nc_k
(h({P_i},\xi_k)-x\cdot\xi_k)^p\\
&\leq\max_{x\in\tau\Omega}\sum_{k=1}^Nc_k
(h({\tau\Omega},\xi_k)-x\cdot\xi_k)^p+1,
\end{align*}
where $\Omega=\{x:x\cdot\xi_k\leq1,\ k=1,...,N\}$ and $\tau$ satisfies $T_q(\tau\Omega)=1$. Let
$$\mathcal{M}=\max_{x\in\tau\Omega}\sum_{k=1}^Nc_k
(h({\tau\Omega},\xi_k)-x\cdot\xi_k)^p+1.$$
Then $\mathcal{M}>0$ is independent of $i$. Hence, for any $i$
$$h({P_i},\xi_k)\leq\bigg(\frac{\mathcal{M}}{\min_{1\leq k\leq N}c_k}\bigg)^{\frac{1}{p}}<\infty,\ \ k=1,...,N,$$
which implies that $\{P_i\}_i$ is bounded.

By Lemma \ref{lem5.1.2} and the Blaschke Selection theorem, there exists a convergent subsequence $\{P_{i_j}\}_j$ of $\{P_{i}\}_i$ such that $P_{i_j}\rightarrow P$.
\end{proof}

Finally, we give the existence of solution to the $L_p$ Minkowski problem $q$-torsional rigidity for discrete measure when $0<p<1$ and $q>1$.

\begin{theorem}\label{the5.3.1}
Let $\mu$ be a finite positive Borel measure on $\mathbb{S}^{n-1}$
which is not concentrated on any closed hemisphere. If $0<p<1$ and $q>1$. Then, for $\mu=\sum_{k=1}^Nc_k\delta_{\xi_k}$,
there exists a polytope $P$ containing the origin in their
interior such that $\mu_{p,q}^{tor}(P,\cdot)=\mu$.
\end{theorem}

\begin{proof}
For the discrete measure $\mu$, by the Lemma \ref{lem5.3.2}, there exists a  polytope $\Omega_0$ which solves problem (\ref{5.2.1}), that is $T_q(\Omega_0)=1$ and
$$\Phi(\Omega_0,x_{\Omega_0})=\inf\{\max_{x\in\Omega}
\Phi(\Omega,x): \Omega\in \mathcal{P}(\xi_1,...,\xi_N),\ T_q(\Omega)=1\}.$$
By the Lemma \ref{lem5.1.3}, then $P_0=\Omega_0-x_{\Omega_0}$ is still the solution to the problem (\ref{5.2.1}) and $x_{P_0}=o$.
This together with the Lemma \ref{lem5.3.1}, Lemma \ref{lem5.2.1} and Lemma \ref{lem5.2.2}, we have
$$\mu_{p,q}(\lambda_0P_0,\cdot)=\mu,$$
where $\lambda_0=\bigg(\frac{q-1}{q+n(q-1)}
\sum_{k=1}^Nc_kh(P,\xi_k)^p\bigg)^
{\frac{1}{q+n(q-1)-p}}$. Namely $P=\lambda_0P_0$ is the desired solution.
\end{proof}
\

\section{ The $L_p$ Minkowski problem of $q$-torsional rigidity for general measure}

Let $\Omega$ be an open bounded convex set of $\mathbb{R}^n$ and $\varphi$ be the solution of (\ref{1.2}) in $\Omega$.
Let $M_{\Omega}=\max_{\bar{\Omega}}\varphi$, for every $t\in[0,M_\Omega]$, we define
$$\Omega_t=\{x\in\Omega: \varphi(x)>t\}.$$
By the general concavity theorems in \cite{Ke,Ko}, $\Omega_t$ is convex for every $t$. Moreover, $\nabla\varphi=0$ if and only if $\varphi(x)=M_\Omega$ such that
$$\partial\Omega_t=\{x\in\Omega: \varphi(x)=t\},\ \ t\in(0,M_\Omega).$$
The following lemma shows an $L^{\infty}$ estimate for the gradient of $\varphi$.

\begin{lemma}\label{lem6.1}
Let $\Omega$ be an open bounded convex subset of $\mathbb{R}^n$ and let $\varphi$ be the solution of problem (\ref{1.2}) in $\Omega$, then for every $x\in\Omega$
$$|\nabla\varphi(x)|\leq diam(\Omega).$$
\end{lemma}

\begin{proof}
Let $x^{\prime}\in\Omega$ and $t=\varphi(x^{\prime})>0$. If $\varphi(x^{\prime})=M_\Omega$, then $\nabla \varphi(x^{\prime})=0$ and the claim is true.

Assume $\varphi(x^{\prime})<M_\Omega$, this implies that $x^{\prime}\in\partial\Omega_t$. The convex set $\Omega_t$ admits a support hyperplane $H$ at $x^{\prime}$. We may choose an orthogonal coordinate system with origin $o$ and coordinates $x_1,...,x_n$, in $\mathbb{R}^n$, such that $x^{\prime}=o$, $H=\{x\in\mathbb{R}^n: x_n=0\}$ and $\Omega_t\subset\{x\in\mathbb{R}^n: x_n\geq0\}$. By a standard argument based on the implicit function theorem, $\partial M_t$ is of class $C^{\infty}$ such that $H$ is in fact the tangent hyperplane to $\partial M_t$ at $x^{\prime}$. Consequently we have
$$|\nabla\varphi(x^{\prime})|=\frac{\partial\varphi}{\partial x_n}(x^{\prime}).$$
We also have the inclusion $\Omega_t\subset\{x\in\mathbb{R}^n: x_n\leq d\}$, where $d=diam (\Omega)$. Let us introduce the function
$$\psi(x)=\psi(x_1,...,x_n)=t+x_n(d-x_n), \ \ x\in\mathbb{R}^n. $$
Note that $\triangle_q\psi(x)=-1$ for every $x\in\mathbb{R}^n$ and $\psi(x)\geq t$ for $x\in\{x=(x_1,...,x_n)\in\mathbb{R}^n: 0\leq x_n\leq d\}$. In particular $\psi\geq\varphi$  on $\partial\Omega_t$, and by the Comparison Principle,
$$\psi(x)\geq\varphi(x), \ \ x\in\Omega_t.$$
Finally, as $\varphi(x^{\prime})=\psi(x^{\prime})$,
$$\frac{\partial\varphi}{\partial x_n}(x^{\prime})\leq\frac{\partial\psi}{\partial x_n}(x^{\prime})=d.$$
\end{proof}

\begin{theorem}\label{the6.1}
Let $\mu$ be a finite positive Borel measure on $\mathbb{S}^{n-1}$
which is not concentrated on any closed hemisphere.
If $0<p<1$ and $q>1$. Then
there exists a convex body $K\in\mathcal{K}_o^n$ such that
$$\mu_{p,q}^{tor}(K,\cdot)=\mu.$$
\end{theorem}

\begin{proof}
For the given measure $\mu$ satisfies the assumptions in the theorem,
there exists a sequence of discrete measure $\mu_i$ defined on $\mathbb{S}^{n-1}$ whose support is not contained in a closed hemisphere so that $\mu_i\rightarrow\mu$ weakly as $i\rightarrow\infty$ (see the proof of Theorem 7.1.2 in \cite{S}).

From the Theorem \ref{the5.3.1}, for each $\mu_i$, there are polytopes $P_i$ containing the origin in their interior, such that, for $i\geq1$,
\begin{align}\label{6.1}
\mu_i=h(P_i,\cdot)^{1-p}\mu_{q}^{tor}(P_i,\cdot)
=\mu_{p,q}^{tor}(P_i,\cdot).
\end{align}

Now we show that the sequence of $\{P_i\}$ is bounded. Let $R_i:=\max\{h(P_i,v): v\in\mathbb{S}^{n-1}\}$ and choose $v_0\in\mathbb{S}^{n-1}$ such that $R_i=h(P_i,v_0)$. Then $[0,R_iv_i]\subset P_i$, and thus $R_i\langle u,v_0\rangle_+\leq h(P_i,u)$ for $u\in\mathbb{S}^{n-1}$. Hence, by the argument of Lemma 2.3 in \cite{HLYZ}, we have
$$R_i\leq c_0$$
for some constant $c_0>0$. This shows that the sequence of $\{P_i\}$ is bounded. By the Blaschke's selection theorem, there exists a sunsequence of $\{P_i\}$ which converges to a compact convex set $K$. Obviously, $o\in K$ since $o\in P_i$. From the Lemma \ref{lem6.1} and the definition of $T_q(K)$, we have
$$V(P_i)^{q-1}\geq\frac{1}{(diam(P_i))^{q(q-1)}}
T_q(P_i)\geq\frac{1}{d^{q(q-1)}}T_q(B_2^n)>0$$
for some $d>0$. This implies that $V(K)>0$ due to $\lim_{i\rightarrow\infty}P_i=K$. Thus $K$ is a convex body containing the origin in its interior.

Next, we show that the $K$ is the desired solution.
For any continuous function $f\in C(\mathbb{S}^{n-1})$, by (\ref{6.1}), one has
$$\int_{\mathbb{S}^{n-1}}\frac{f(u)}{h(P_i,u)^{1-p}}d\mu_i
=\int_{\mathbb{S}^{n-1}}f(u)d\mu_q^{tor}(P_i,u).$$
Since $h(P_i,\cdot)\rightarrow h(K,\cdot)$ uniformly on $\mathcal{S}^{n-1}$,  $\mu_i\rightarrow\mu$ weakly as $i\rightarrow\infty$, and $\mu_q^{tor}(P_i,\cdot)\rightarrow\mu_q^{tor}(K,\cdot)$ weakly as $i\rightarrow\infty$, then, there is
$$\int_{\mathbb{S}^{n-1}}\frac{f(u)}{h(K,u)^{1-p}}d\mu
=\int_{\mathbb{S}^{n-1}}f(u)d\mu_q^{tor}(K,u)$$
Since $f\in C(\mathbb{S}^{n-1})$ is arbitrary, thus
$$\mu=\mu_{p,q}^{tor}(K,\cdot).$$
The theorem is proved.
\end{proof}
\
\

\noindent{\bf Acknowledgments}\\

The authors would like to express their heartfelt thanks to the editor and referees for helpful comments and suggestions.\\

\end{document}